\documentclass[reqno]{amsart}
\usepackage{relsize}
\usepackage{amsmath}
\usepackage{amsthm}
\usepackage{amssymb}
\usepackage{hyperref}
\usepackage{comment}
\usepackage{xcolor}
\usepackage{mathrsfs}
\usepackage{todonotes}

\newcommand{\REG}{{\rm REG}}

\newcommand{\GCH}{{\rm GCH}}

\newcommand{\ZFC}{{\rm ZFC}}

\renewcommand{\P}{{\mathbb P}}

\newcommand{\C}{{\mathbb C}}
\newcommand{\Q}{{\mathbb Q}}
\newcommand{\R}{{\mathbb R}}
\renewcommand{\S}{\mathbb{S}}

\renewcommand{\H}{{\mathbb H}}

\newcommand{\Add}{\mathop{\rm Add}}

\newcommand{\forces}{\Vdash}
\newcommand{\forced}{\Vdash}

\newcommand{\restrict}{\upharpoonright}

\newcommand{\<}{\langle}
\renewcommand{\>}{\rangle}

\newcommand{\st}{:}

\newcommand{\supp}{\mathop{\rm supp}}

\newcommand{\dom}{\mathop{\rm dom}}

\newcommand{\range}{\mathop{\rm range}}

\newcommand{\Lim}{{\mathop{\rm Lim}}}

\newcommand{\cf}{\mathop{\rm cf}}

\newcommand{\NS}{{\mathop{\rm NS}}}

\newcommand{\Refl}{{\mathop{\rm Refl}}}
\newcommand{\Ind}{{\mathop{\rm Ind}}}
\newcommand{\Tr}{{\mathop{\rm Tr}}}
\newcommand{\wc}{{\mathop{\rm wc}}}
\newcommand{\WC}{{\mathop{\rm WC}}}

\newtheorem{theorem}{Theorem}[section]
\newtheorem{lemma}[theorem]{Lemma}
\newtheorem{corollary}[theorem]{Corollary}

\theoremstyle{definition}
\newtheorem{question}[theorem]{Question}
\newtheorem{remark}[theorem]{Remark}

\newtheorem{definition}[theorem]{Definition}

\date{\today}

\begin{document}

\title{Adding a non-reflecting weakly compact set}

\author[Brent Cody]{Brent Cody}
\address[Brent Cody]{ 
Virginia Commonwealth University,
Department of Mathematics and Applied Mathematics,
1015 Floyd Avenue, PO Box 842014, Richmond, Virginia 23284, United States
} 
\email[B. ~Cody]{bmcody@vcu.edu} 
\urladdr{http://www.people.vcu.edu/~bmcody/}

\begin{abstract}
  For $n<\omega$, we say that the \emph{$\Pi^1_n$-reflection principle holds at $\kappa$} and write $\Refl_n(\kappa)$ if and only if $\kappa$ is a $\Pi^1_n$-indescribable cardinal and every $\Pi^1_n$-indescribable subset of $\kappa$ has a $\Pi^1_n$-indescribable proper initial segment. The $\Pi^1_n$-reflection principle $\Refl_n(\kappa)$ generalizes a certain stationary reflection principle and implies that $\kappa$ is $\Pi^1_n$-indescribable of order $\omega$. We define a forcing which shows that the converse of this implication can be false in the case $n=1$; that is, we show that $\kappa$ being $\Pi^1_1$-indescribable of order $\omega$ need not imply $\Refl_1(\kappa)$. Moreover, we prove that if $\kappa$ is $(\alpha+1)$-weakly compact where $\alpha<\kappa^+$, then there is a forcing extension in which there is a weakly compact set $W\subseteq\kappa$ having no weakly compact proper initial segment, the class of weakly compact cardinals is preserved and $\kappa$ remains $(\alpha+1)$-weakly compact. We also formulate several open problems and highlight places in which standard arguments seem to break down.
\end{abstract}



\maketitle





%
%
%

\section{Introduction}\label{sectionintroduction}

For a regular cardinal $\kappa$, we say that a set $S\subseteq\kappa$ is \emph{$\Pi^1_n$-indescribable} if for every $A\subseteq V_\kappa$ and every $\Pi^1_n$-formula $\varphi$, whenever $(V_\kappa,\in,A)\models \varphi$ there exists $\alpha\in S$ such that $(V_\alpha,\in,A\cap V_\alpha)\models \varphi$. Levy proved \cite{MR0281606} that if $\kappa$ is $\Pi^1_n$-indescribable then the collection $\Pi^1_n(\kappa)=\{X\subseteq\kappa\st \text{$X$ is not $\Pi^1_n$-indescribable}\}$ is a normal proper ideal on $\kappa$. The work of Hellsten (see \cite{MR2026390}, \cite{MR2252250} and \cite{MR2653962}) as well as the recent work of Bagaria-Magidor-Sakai \cite{MR3416912} has shown that various results involving the nonstationary ideal can be extended to the $\Pi^1_n$-indescribable ideal $\Pi^1_n(\kappa)$. For example, the notion of $\alpha$-Mahloness of a cardinal $\kappa$ where $\alpha\leq\kappa^+$ can be generalized to that of $\alpha$-$\Pi^1_n$-indescribability where $\alpha\leq\kappa^+$ (see Definition \ref{definition_gamma_indescribable} below). Hellsten proved \cite{MR2252250} that the $\Pi^1_n$-indescribability ideal $\Pi^1_n(\kappa)$ is not $\kappa^+$-saturated if $\kappa$ is $\kappa^+$-$\Pi^1_n$-indescribable. This is analogous to a result of Baumgartner, Taylor and Wagon \cite{MR0505505} stating that $\NS_\kappa\restrict\REG$ is not $\kappa^+$-saturated if $\kappa$ is $\kappa^+$-Mahlo. Extending work of Jech and Woodin \cite{MR805967} on the nonstationary ideal, Hellsten proved \cite{MR2653962} that it is consistent relative to the existence of a measurable cardinal $\kappa$ of Mitchell order $\alpha<\kappa^+$ that there is an $\alpha$-$\Pi^1_1$-indescribable cardinal $\kappa$ such that the $\Pi^1_1$-indescribable ideal $\Pi^1_1(\kappa)$ is $\kappa^+$-saturated. Recall that for a regular cardinal $\kappa$, we say that the stationary reflection principle $\Refl(\kappa)$ holds if and only if every stationary subset of $\kappa$ has a stationary proper initial segment. Bagaria-Magidor-Sakai \cite{MR3416912} generalized Jensen's result \cite{MR0309729} which states that in $L$, the weakly compact cardinals are precisely the regular cardinals at which the stationary reflection principle holds, by providing a similar characterization of the $\Pi^1_n$-indescribable cardinals in $L$. Additionally, extending Solovay's theorem on splitting stationary sets, Hellsten \cite[Theorem 2]{MR2653962} has shown (for a proof see \cite[Proposition 6.4]{MR2768692}) that the $\Pi^1_n$-indescribable ideal on a $\Pi^1_n$-indescribable cardinal $\kappa$ is nowhere $\kappa$-saturated. It is important to point out that for many reasons, these generalizations require substantial effort: for example, in many situations, one must deal with the fact that a non-weakly compact set (i.e. a non-$\Pi^1_1$-indescribable set) can become weakly compact in a forcing extension,\footnote{Indeed, Kunen showed \cite{MR495118} that a non-weakly compact cardinal can become supercompact in a forcing extension by showing that the forcing to add a Cohen subset to $\kappa$ is equivalent to a two-step iteration $\mathbb{S}*\dot{T}$ in which the first step  $\mathbb{S}$ is a forcing which adds homogeneous Suslin tree $\dot{T}$, and then in the second step one forces with this tree.} whereas a nonstationary set cannot be forced to become stationary. 

In this article we continue this line of research by generalizing the stationary reflection principle $\Refl(\kappa)$ as follows. The stationary reflection principle $\Refl(\kappa)$ is formulated by referencing the nonstationary ideals $\NS_\gamma$ for $\gamma\leq\kappa$ with $\cf(\gamma)>\omega$, and one may formulate new reflection principles by replacing these ideals with others. Let $\Refl_0(\kappa)$ be the statement asserting that $\kappa$ is inaccessible and for every stationary $S\subseteq\kappa$ there is an inaccessible cardinal $\gamma<\kappa$ such that $S\cap\gamma$ is a stationary subset of $\gamma$. Since a set $S\subseteq\gamma$ is $\Pi^1_0$-indescribable if and only if $\gamma$ is inaccessible and $S$ is stationary \cite{MR2252250}, we obtain a direct generalization of $\Refl_0(\kappa)$ as follows. For $n<\omega$, we say that the \emph{$\Pi^1_n$-reflection principle} holds at $\kappa$ and write $\text{Refl}_n(\kappa)$ if and only if $\kappa$ is a $\Pi^1_n$-indescribable cardinal and for every $\Pi^1_n$-indescribable set $S\subseteq\kappa$ there exists a $\gamma<\kappa$ such that $S\cap\gamma$ is $\Pi^1_n$-indescribable. For the case in which $n=1$, we say that the \emph{weakly compact reflection principle} holds at $\kappa$ and write $\Refl_{wc}(\kappa)$ if and only if $\kappa$ is a weakly compact cardinal and every weakly compact subset of $\kappa$ has a weakly compact initial segment. Since $\Refl_n(\kappa)$ holds whenever $\kappa$ is $\Pi^1_{n+1}$-indescribable (see the beginning of Section \ref{section_preliminaries} below), and since $\Refl_n(\kappa)$ implies that there are many $\Pi^1_n$-indescribable cardinals below $\kappa$, it follows that the consistency strength of ``$\exists\kappa$ $\Refl_n(\kappa)$'' is strictly greater than the existence of a $\Pi^1_n$-indescribable cardinal but not greater than the existence of a $\Pi^1_{n+1}$-indescribable cardinal. The exact consistency strength of ``$\exists\kappa$ $\Refl_n(\kappa)$'' is not known, but the work of Mekler and Shelah \cite{MR1029909} leads to a conjecture in Section \ref{section_questions} below. Additionally, one may derive the consistency of ``$\exists\kappa \forall n<\omega$ $\Refl_n(\kappa)$'' from the existence of a $\Pi^1_n$-indescribable cardinal of order $\omega$ (see Definition \ref{definition_indescribable_of_order_gamma} below).

In Lemma \ref{lemma_weakly_compacts_are_weakly_compact} below, we observe that $\Refl_n(\kappa)$ implies that $\kappa$ is $\Pi^1_n$-indescribable of order $\omega$. The main result of this article addresses the question: if $\kappa$ is $\Pi^1_n$-indescribable of high order $\gamma<\kappa^+$, does this entail that $\Refl_n(\kappa)$ holds? In the case where $n=1$ we show that the answer to this question is consistently, no.\footnote{In fact, this question can also be answered by using the result of Bagaria-Magidor-Sakai \cite{MR3416912} mentioned above. Although the characterization of $\Pi^1_{n+1}$-indescribable cardinals in $L$ given by Bagaria-Magidor-Sakai does not resemble the reflection principles $\Refl_n(\kappa)$ considered here, Sakai informed the author, after the current article was written, that in $L$, the Bagaria-Magidor-Sakai characterization is equivalent to $\Refl_n(\kappa)$. Bagaria-Magidor-Sakai showed that in $L$, a cardinal $\kappa$ is $\Pi^1_{n+1}$-indescribable if and only if $\Refl_{n}(\kappa)$ holds. Thus, in $L$, $\Refl_n(\kappa)$ fails at the least greatly $\Pi^1_n$-indescribable cardinal.} Suppose $\kappa$ is a weakly compact cardinal. We say that $S\subseteq\kappa$ is a \emph{nonreflecting weakly compact subset of $\kappa$} if and only if $S$ is a weakly compact set having no weakly compact proper initial segment; similarly, one can define the notion of \emph{nonreflecting $\Pi^1_n$-indescribable subset of $\kappa$}. The following theorem establishes the consistency of the failure of the weakly compact reflection principle at $\kappa$ with $\kappa$ being weakly compact of any order $\gamma<\kappa^+$.

\begin{theorem}\label{theorem_main_theorem}
Suppose that $\Refl_\wc(\kappa)$ holds, $\kappa$ is weakly compact of order $(\alpha+1)$ where $\alpha<\kappa^+$ and $\GCH$ holds. Then there is a cofinality-preserving forcing extension in which there is a nonreflecting weakly compact subset of $\kappa$, the class of weakly compact cardinals is preserved and $\kappa$ remains $(\alpha+1)$-weakly compact.
\end{theorem}

Notice that the assumption $\Refl_\wc(\kappa)$ allows us to avoid trivialities, since if $\Refl_\wc(\kappa)$ fails then $\kappa$ already has a nonreflecting weakly compact subset. The forcing used in the proof of Theorem \ref{theorem_main_theorem}, adds a nonreflecting weakly compact subset of $\kappa$. Notice that many forcings throughout the literature also add nonreflecting weakly compact subsets to a weakly compact cardinal $\kappa$, such as an Easton-support iteration to turn $\kappa$ into the least weakly compact cardinal by adding nonreflecting stationary subsets to each weakly compact $\gamma<\kappa$; the new features of the forcing developed here is that it preserves the class of weakly compact cardinals and also preserves the fact that $\kappa$ is $(\alpha+1)$-weakly compact. Our proof of Theorem \ref{theorem_main_theorem}, relies heavily on the simple characterization of weak compactness in terms of elementary embeddings, which is reviewed in Lemma \ref{lemma_characterizations_of_weak_compactness}, below. Although one can characterize $\Pi^1_n$-indescribability in terms of elementary embeddings \cite{MR1133077}, it is not known whether or not the techniques used in our proof of Theorem \ref{theorem_main_theorem} generalize.

Another important feature of the proof of Theorem \ref{theorem_main_theorem}, is that we will be concerned with showing that various forcings do not create new instances of weak compactness, a feature which is not present when dealing with forcing to add a nonreflecting stationary set. Our main tool in this regard, will be Hamkins' work showing that extensions with the approximation and cover properties have no new large cardinals \cite{MR2063629}; see Lemma \ref{lemma_approximation_and_cover} below.

Many questions about the weakly compact and $\Pi^1_n$-indescribable reflection principles remain open, some of which are discussed in Section \ref{section_questions} below. We highlight places in which standard arguments seem to break down in this context.

\section{Preliminaries}\label{section_preliminaries}



Since there is a $\Pi^1_{n+1}$-sentence $\varphi$ such that for any $\gamma$ a set $S\subseteq\gamma$ is $\Pi^1_n$-indescribable if and only if $(V_\gamma,\in,S)\models\varphi$ \cite[Corollary 6.9]{Kanamori:Book}, it follows that $\Refl_n(\kappa)$ holds if $\kappa$ is $\Pi^1_{n+1}$-indescribable. Therefore ``$\exists\kappa\ \Refl_n(\kappa)$'' is consistence relative to the existence of a $\Pi^1_{n+1}$-indescribable cardinal.

Assuming $\kappa$ is a $\Pi^1_n$-indescribable cardinal, the $\Pi^1_n$-ideal 
\[\Pi^1_n(\kappa)=\{X\subseteq\kappa\st \text{$X$ is not $\Pi^1_n$-indescribable}\}\]
is a normal ideal on $\kappa$ \cite{MR0281606}. We will denote the corresponding collection of positive sets by
\[\Pi^1_n(\kappa)^+=\{X\subseteq\kappa\st \text{$X$ is $\Pi^1_n$-indescribable}\}\]
and the dual filter is written as
\[\Pi^1_n(\kappa)^*=\{X\subseteq\kappa\st \text{$\kappa\setminus X$ is not $\Pi^1_n$-indescribable}\}.\] In the case where $n=0$, we obtain the nonstationary ideal $\text{NS}_\kappa=\Pi^1_0(\kappa)$. Thus, in this case, the filter $\Pi^1_0(\kappa)^*$ is generated by the collection of club subsets of $\kappa$ and $\Pi^1_0(\kappa)^+$ is the collection of stationary subsets of $\kappa$. Furthermore, a cardinal $\kappa$ is $\Pi^1_0$-indescribable if and only if $\kappa$ is inaccessible \cite[Proposition 6.3]{Kanamori:Book}. In the case where $n=1$, we sometimes refer to $\Pi^1_1(\kappa)$ as the \emph{weakly compact ideal}. For $0<n<\omega$ one can characterize a natural collection of subsets of $\kappa$ which generates the filter $\Pi^1_n(\kappa)^*$ as follows. We say that a set $C\subseteq\kappa$ is \emph{$0$-club} if and only if $C$ is club. Furthermore, for $n>0$, $C$ is said to be \emph{$n$-club} if $C\in\Pi^1_{n-1}(\kappa)^+$ and whenever $C\cap\alpha\in\Pi^1_{n-1}(\alpha)^+$ we have $\alpha\in C$. Note that, $C$ is $1$-club if and only if $C$ is a stationary subset of $\kappa$ which contains all of its inaccessible stationary reflection points. Under the assumption that $\kappa$ is a $\Pi^1_n$-indescribable cardinal, it follows from the work of Sun \cite{MR1245524} and Hellsten \cite{MR2252250} that $S\in\Pi^1_n(\kappa)^+$ if and only if for every $n$-club $C\subseteq\kappa$ we have $S\cap C\neq\emptyset$. Thus, if $\kappa$ is a $\Pi^1_n$-indescribable cardinal, a set $X\subseteq\kappa$ is in the filter $\Pi^1_n(\kappa)^*$ if and only if $X$ contains an $n$-club.

For $n<\omega$, Hellsten defined an operation $\Tr_n$ on $\mathscr{P}(\kappa)$  analogous to Mahlo's operation\footnote{Hellsten denoted this operation by $M_n$, but here we use $\Tr_n$ to avoid notational confusion since we use $M$ to denote a $\kappa$-model.} by
\[\Tr_n(X)=\{\alpha<\kappa\st X\cap\alpha\in\Pi^1_n(\alpha)^+\},\]
as well as a transitive wellfounded partial order $<^n$ on $\Pi^1_n(\kappa)^+$ analogous to Jech's ordering on stationary sets; for $S,T\in\Pi^1_n(\kappa)^+$ we have
\[S<^n T \iff T\setminus \Tr_n(S)\in \Pi^1_n(\kappa).\]
Hellsten also defined the \emph{order of a $\Pi^1_n$-indescribable set $S\in\Pi^1_n(\kappa)^+$} to be its rank under $<^n$,
\[o_n(S)=\sup\{o_n(X)+1\st \text{$X<^nS$ and $X\in\Pi^1_n(\kappa)^+$}\};\]
and the \emph{order of a $\Pi^1_n$-indescribable cardinal $\kappa$} to be the height of $<^n$,
\[o_n(\kappa)=\sup\{o_n(S)+1\st S\in\Pi^1_n(\kappa)^+\}.\]

Generalizing the notion of great Mahloness from \cite{MR0505505}, Hellsten provided the following (for details see \cite[Definition 2 and Lemma 10]{MR2252250}).

\begin{definition}[\cite{MR2252250}] A cardinal $\kappa$ is \emph{greatly $\Pi^1_n$-indescribable} if and only if $o_n(\kappa)=\kappa^+$.
\end{definition} 

We will find it useful to consider orders of $\Pi^1_n$-indescribability less than $\kappa^+$, generalizing the Mahlo-hierarchy.
\begin{definition}\label{definition_indescribable_of_order_gamma}
Given cardinal $\kappa$ and an ordinal $0<\gamma<\kappa^+$, we say that $\kappa$ is \emph{$\Pi^1_n$-indescribable of order $\gamma$} if and only if $o_n(\kappa)\geq\gamma$.
\end{definition}

In what follows we find it easier to work with another, slightly more concrete, characterization of the order of a $\Pi^1_n$-indescribable cardinal, which we describe in Lemma \ref{lemma_characterization_of_order} below. Working towards this characterization, let us establish a few basic properties of the operation $\Tr_n$, which are implicit in \cite{MR2252250}. Hellsten proved that $\Tr_n$ is a \emph{generalized Mahlo operation} for the $\Pi^1_n$-ideal over a regular cardinal $\kappa$. Among other things, this implies \cite[Theorem 2]{MR2252250} that the $\Pi^1_n$-indescribable filter is closed under the operation $\Tr_{n-1}$, which is analogous to the fact that the club filter is closed under the map taking a set $X\in\mathscr{P}(\kappa)$ to its set of limit points $X'$. As an easy consequence, we obtain the following.

\begin{lemma}\label{lemma_reflection_points_of_a_measure_zero_set}
Suppose $\kappa$ is a $\Pi^1_n$-indescribable cardinal. If $Z\in\Pi^1_n(\kappa)$ then $\Tr_n(Z)\in\Pi^1_n(\kappa)$.
\end{lemma}

\begin{proof}
Suppose $Z\in \Pi^1_n(\kappa)$ and $\Tr_n(Z)\notin\Pi^1_n(\kappa)$. Let $C\subseteq\kappa$ be an $n$-club with $C\cap Z=\emptyset$. Since the filter $\Pi^1_n(\kappa)^*$ is closed under $\Tr_{n-1}$, it follows that 
\[\Tr_{n-1}(C)=\{\alpha<\kappa\st \text{$\alpha$ is $\Pi^1_{n-1}$-indescribable and $C\cap\alpha\in \Pi^1_{n-1}(\alpha)^+$}\}\]
contains an $n$-club subset of $\kappa$ and since $C$ is $n$-club we have $\Tr_{n-1}(C)\subseteq C$. Since $\Tr_n(Z)\in\Pi^1_n(\kappa)^+$ we may choose $\eta\in \Tr_n(Z)\cap \Tr_{n-1}(C)$. Notice that $Z\cap\eta\in\Pi^1_n(\eta)^+$ and $C\cap \eta$ is an $n$-club subset of $\eta$ disjoint from $Z\cap\eta$, a contradiction.
\end{proof}

Following Hellsten's notation, elements of the boolean algebra $\mathscr{P}(\kappa)/\Pi^1_n(\kappa)$ are written as $[S]_n$ where $S\in\Pi^1_n(\kappa)^+$ and $[S]_n$ is the equivalence class of $S$ modulo the ideal $\Pi^1_n(\kappa)$. We let $\leq_n$ denote the usual ordering on $\mathscr{P}(\kappa)/\Pi^1_n(\kappa)$. Hellsten proved \cite[Lemma 6]{MR2252250} that $\Tr_n$ is well defined as a map $\Tr_n:\mathscr{P}(\kappa)/\Pi^1_n(\kappa)\rightarrow \mathscr{P}(\kappa)/\Pi^1_n(\kappa)$ where $\Tr_n([S]_n)=[\Tr_n(S)]_n$, and we observe that this follows directly from Lemma \ref{lemma_reflection_points_of_a_measure_zero_set}. 

\begin{corollary}\label{corollary_Mahlo_respects_subsets_mod_ideal}
Suppose $\kappa$ is a $\Pi^1_n$-indescribable cardinal and $S,T\in \mathscr{P}(\kappa)$. If $S\leq_n T$ then $\Tr_n(S)\leq_n \Tr_n(T)$. Thus, $\Tr_n$ is well defined as a map $\mathscr{P}(\kappa)/\Pi^1_n(\kappa)\to \mathscr{P}(\kappa)/\Pi^1_n(\kappa)$.
\end{corollary}


It is well known (see \cite{MR0505505} and \cite[Lemma 3]{MR2252250}), that if $I$ is a normal ideal on $\kappa$ then the diagonal intersection of a collection of ${\leq}\kappa$-many subsets of $\kappa$ is independent of the indexing used, modulo $I$, and this allows one to calculate $\kappa^+$ iterates of $\Tr_n$ as an operation on $\mathscr{P}(\kappa)/\Pi^1_n(\kappa)$. For the readers convenience let us briefly recall how this is done, in the context of the $\Pi^1_n$-indescribable ideal. Given a collection $\{[A_i]_n\st i<\beta\}\subseteq\mathscr{P}(\kappa)/\Pi^1_n(\kappa)$ where $\beta<\kappa^+$, let $f:\beta\to\kappa$ be injective and define $g:\kappa\to\mathscr{P}(\kappa)$ by $g(\alpha)=\kappa$ if $\alpha\notin f[\beta]$ and $g(\alpha)=A_i$ if $f(i)=\alpha$. We define
\[\bigtriangleup\{[A_i]_n\st i<\beta\}=[\textstyle{\bigtriangleup} g]_n\]
where, as usual, $\bigtriangleup g=\{\xi<\kappa\st (\forall i<\xi)\ \xi\in g(i)\}$.  For $A\in\Pi^1_n(\kappa)^+$, we define a sequence $\<\Tr^\alpha_n([A]_n)\st \alpha<\kappa^+\>$ in the boolean algebra $\mathscr{P}(\kappa)/\Pi^1_n(\kappa)$:
\begin{align*}
\Tr^0_n([A]_n)&=[A]_n\\
\Tr^{\alpha+1}_n([A]_n)&=\Tr_n(\Tr^\alpha_n([A]_n)),\\
\Tr^\gamma_n([A]_n)&= \bigtriangleup\{\Tr^\alpha_n([A]_n) \st\alpha<\gamma\}\text{ if $\gamma<\kappa^+$ is a limit.}  
\end{align*}
Similarly, one may iterate $\Tr_n$ as an operation on $\mathscr{P}(\kappa)$, $\kappa$-many times. Notice that for limits $\gamma<\kappa$, if we let $A_i$ be a representative of the equivalence class $\Tr^i_n([\kappa]_n)$ for $i<\gamma$, then $\Tr^\gamma_n([A]_n)=[\bigcap\{A_i\st i<\gamma\}]_n$. Observe that $\Tr_n(\kappa)$ is the set of $\Pi^1_n$-indescribable cardinals below $\kappa$, and $\Tr_n^2(\kappa)$ is the set of $\Pi^1_n$-indescribable cardinals $\mu<\kappa$ such that $\Tr_n(\kappa)\cap\mu$ is a $\Pi^1_n$-indescribable subset of $\mu$. 

\begin{definition}\label{definition_gamma_indescribable}
For $\gamma\leq\kappa^+$, we say that $\kappa$ is \emph{$\gamma$-$\Pi^1_n$-indescribable} if and only if $\Tr_n^\alpha([\kappa])_n> 0$ for all $\alpha<\gamma$. 
\end{definition}

Thus $\kappa$ is $1$-$\Pi^1_n$-indescribable if and only if $\kappa$ is $\Pi^1_n$-indescribable, $\kappa$ is $2$-$\Pi^1_n$-indescribable if and only if the collection of $\Pi^1_n$-indescribable cardinals less than $\kappa$ is a $\Pi^1_n$-indescribable subset of $\kappa$, etc.

\begin{lemma}
For $\alpha<\kappa$ we have $\Tr^\alpha_n([\kappa]_n)=[\{\gamma<\kappa\st\text{$\gamma$ is $\alpha$-$\Pi^1_n$-indescribable}\}]_n$.
\end{lemma}


We now show that these two concepts of orders of indescribability are equivalent.

\begin{lemma}\label{lemma_characterization_of_order}
Suppose $\kappa$ is a $\Pi^1_n$-indescribable cardinal and $\gamma\leq\kappa^+$. Then $\kappa$ is $\Pi^1_n$-indescribable with order $\gamma$ if and only if $\kappa$ is $\gamma$-$\Pi^1_n$-indescribable. In other words, $o_n(\kappa)\geq\gamma$ if and only if for all $\alpha<\gamma$ we have $\Tr_n^\alpha([\kappa])>0$ in the boolean algebra $P(\kappa)/\Pi^1_n(\kappa)$.
\end{lemma}

\begin{proof}
The reverse direction follows easily from the fact that if $S\in\Pi^1_n(\kappa)^+$ then $S<^n \Tr_n(S)$.

For the forward direction, suppose $\kappa$ is $\Pi^1_n$-indescribable with order $\gamma<\kappa^+$. Then for each $\delta<\gamma$ there exists a $\Pi^1_n$-indescribable set $X\in \Pi^1_n(\kappa)^+$ such that $o_n(X)=\delta$. This implies that there is a strictly increasing chain $\<X_\alpha\st \alpha<\delta\>$ in the poset $(\Pi^1_n(\kappa)^+,<^n)$ below $X$; thus $\alpha<\beta<\delta$ implies $X_\beta\setminus \Tr_n(X_\alpha)\in \Pi^1_n(\kappa)$. We must show that in the boolean algebra $\mathscr{P}(\kappa)/\Pi^1_n(\kappa)$, we have $\Tr^\alpha_n([\kappa]_n)>0$ for all $\alpha<\delta$.

For each $\alpha<\delta$, let $A_\alpha$ be a representative of the equivalence class $\Tr^\alpha_n([\kappa]_n)$. We will use induction to prove that $[X_\alpha]_n\leq_n [A_\alpha]_n$ for all $\alpha<\delta$. It will suffice to inductively construct a sequence $\vec{Z}=\<Z_\alpha\st \alpha <\delta\>$ of sets in $\Pi^1_n(\kappa)$ such that for each $\alpha<\delta$ we have $X_\alpha\subseteq A_\alpha\cup Z_\alpha$. Clearly $[X_0]_n\leq_n[\kappa]_n$ and we let $Z_0=X_0\setminus\kappa=\emptyset$. 

If $\alpha=\beta+1$ is a successor, then we have $X_{\beta+1}\leq_n \Tr_n(X_\beta)$ and by the inductive hypothesis $[X_\beta]_n\leq_n [A_\beta]_n$. By Corollary \ref{corollary_Mahlo_respects_subsets_mod_ideal}, we have $\Tr_n([X_\beta])\leq_n \Tr_n([A_\beta]_n)=\Tr^{\beta+1}_n([\kappa])$ and thus $[X_{\beta+1}]_n\leq_n \Tr^{\beta+1}_n([\kappa]_n)=[A_{\beta+1}]_n$. We let $Z_{\beta+1}=X_{\beta+1}\setminus A_{\beta+1}\in \Pi^1_n(\kappa)$. Notice that $X_{\beta+1}\subseteq A_{\beta+1}\cup Z_{\beta+1}$. 

At limit stages $\alpha<\kappa^+$ we inductively build a sequence $\vec{z}=\<z_\beta\st \beta\leq \alpha\>$ of sets in $\Pi^1_n(\kappa)$ such that for all $\beta\leq \alpha$, $X_\alpha\subseteq A_\beta\cup z_\beta$. Let $z_0=\emptyset$. If $\beta=\xi+1<\alpha$ is a successor, by assumption we have $X_\alpha\leq_n \Tr_n(X_\xi)$ and inductively we have $[X_\xi]_n\leq_n [A_\xi]_n$. By Corollary \ref{corollary_Mahlo_respects_subsets_mod_ideal}, $\Tr_n([X_\xi]_n)\leq_n \Tr_n([A_\xi]_n)=[A_{\xi+1}]_n$, which implies that $[X_\alpha]_n\leq_n [A_{\xi+1}]_n$. Hence there exists $z_{\xi+1}\in \Pi^1_n(\kappa)$ such that $X_\alpha\subseteq A_{\xi+1}\cup z_{\xi+1}$. If $\beta\leq\alpha$ is a limit, fix an injection $f:\beta\to\kappa$ and define two functions $g_A,g_z:\kappa\to\mathscr{P}(\kappa)$ by
\begin{align*}
g_A(\xi)=\begin{cases}
\kappa&\text{if $\xi\notin f[\beta]$}\\
A_i&\text{if $f(i)=\xi$}
\end{cases}
\end{align*}
and
\begin{align*}
g_z(\xi)=\begin{cases}
\kappa&\text{if $\xi\notin f[\beta]$}\\
z_i&\text{if $f(i)=\xi$}
\end{cases}
\end{align*}

Now let $z'_\beta=\bigtriangledown\{z_\xi\st\xi<\beta\}=\bigtriangledown g_z$ and notice that $z'_\beta\in\Pi^1_n(\kappa)$ by normality.  By the inductive hypothesis, $X_\alpha\subseteq A_i\cup z_i$ for all $i<\beta$, and hence
\[
X_\alpha\subseteq \left(\bigtriangleup g_A\right)\cup\left(\bigtriangledown g_z\right) =_n A_\beta\cup z'_\beta.
\]
Thus, we may let $z_\beta\in\Pi^1_n(\kappa)$ be such that $X_\alpha\subseteq A_\beta\cup z_\beta$. This defines the sequence $\vec{z}=\<z_\beta\st \beta\leq \alpha\>$. We let $Z_\alpha=z_\alpha$ and note that $X_\alpha\subseteq A_\alpha\cup Z_\alpha$ holds by construction.
\end{proof}

\section{Basic consequences of the $\Pi^1_n$-reflection principle}

\begin{lemma}\label{lemma_trace_of_wc_is_wc}
Suppose $\Refl_n(\kappa)$ holds and $S\in \Pi^1_n(\kappa)^+$. Then \[\Tr_n(S)=\{\alpha<\kappa\st S\cap\alpha\in \Pi^1_n(\alpha)^+\}\]
is a $\Pi^1_n$-indescribable subset of $\kappa$. 
\end{lemma}
\begin{proof}
Fix an $n$-club $C\subseteq\kappa$. We will show that $\Tr_n(S)\cap C\neq\emptyset$. Since $\kappa$ is a $\Pi^1_n$-indescribable cardinal, the filter $\Pi^1_n(\kappa)^*$ is normal and hence the intersection of fewer than $\kappa$ many $n$-club subsets of $\kappa$ contains an $n$-club. Thus, $S\cap C$ is a $\Pi^1_n$-indescribable subset of $\kappa$. Since $\Refl_n(\kappa)$ holds, there is an $\alpha<\kappa$ such that $S\cap C\cap \alpha\in \Pi^1_n(\alpha)^+$. This implies that $S\cap\alpha\in \Pi^1_n(\alpha)^+$. Furthermore, $C\cap\alpha$ is a $\Pi^1_{n-1}$-indescribable subset of $\alpha$ and thus $\alpha\in C$ because $C$ is $n$-club. Thus $\alpha\in \Tr_n(S)\cap C$.
\end{proof}

The next lemma establishes that $\Refl_n(\kappa)$ implies that $\kappa$ is $\omega$-$\Pi^1_n$-indescribable.

\begin{lemma}\label{lemma_weakly_compacts_are_weakly_compact}
Suppose $\Refl_n(\kappa)$ holds. Then the set 
\[\Ind^1_n\cap\kappa=\{\alpha<\kappa\st\textrm{$\alpha$ is $\Pi^1_n$-indescribable}\}\] is a $\Pi^1_n$-indescribable subset of $\kappa$. Thus, by Lemma \ref{lemma_trace_of_wc_is_wc}, for each $m<\omega$ we have $\Tr_n^m(\Ind^1_n\cap\kappa)\in \Pi^1_n(\kappa)^+$; in other words, $\kappa$ is $\omega$-$\Pi^1_n$-indescribable. 

\end{lemma}

\begin{proof}
Suppose $\Refl_n(\kappa)$ holds. Let $C\subseteq\kappa$ be an $n$-club. We will prove that $(\Ind^1_n\cap\kappa)\cap C\neq\emptyset$. Since $C\in\Pi^1_n(\kappa)^*\subseteq\Pi^1_n(\kappa)^+$ is $\Pi^1_n$-indescribable, $\Refl_n(\kappa)$ implies that there is a $\mu<\kappa$ such that $C\cap\mu\in \Pi^1_n(\mu)^+$, and hence $\mu\in \Ind^1_n\cap\kappa$. Since $C\cap\mu$ is, in particular, a $\Pi^1_{n-1}$-indescribable subset of $\mu$, and since $C$ is an $n$-club we have $\mu\in C$. Thus $\mu\in(\Ind^1_n\cap\kappa)\cap C\neq\emptyset$.
\end{proof}

We now restrict our attention to weak compactness. Below we will write $\Tr_\wc$ to denote the operation $\Tr_1$; that is, if $X\in\mathscr{P}(\kappa)$ then $\Tr_\wc(X)=\{\alpha<\kappa\st X\cap\alpha\in\Pi^1_1(\alpha)^+\}$. In contrast to Lemma \ref{lemma_X_not_weakly_compact_in_N} below, under $\Refl_\wc(\kappa)$ we obtain the following.

\begin{corollary}
Suppose $\Refl_{wc}(\kappa)$ holds. If $S\in \Pi^1_1(\kappa)^+$ then for every $A\subseteq\kappa$ there is a $\kappa$-model $M$ with $\kappa,A,S\in M$ and an elementary embedding $j:M\to N$ with critical point $\kappa$ such that $N\models$ $S\in (\Pi^1_1(\kappa)^+)^N$.
\end{corollary}

\begin{proof}
Suppose $\Refl_{wc}(\kappa)$ holds and $X\in \Pi^1_1(\kappa)^+$. By Lemma \ref{lemma_trace_of_wc_is_wc}, we have $\Tr_\wc(S)\in \Pi^1_1(\kappa)^+$. Fix $A\subseteq\kappa$. There exists a $\kappa$-model $M$ with $\kappa,A,S,\Tr_\wc(S)\in M$ and an elementary embedding $j:M\to N$ with critical point $\kappa$ such that $\kappa\in j(\Tr_\wc(S))$; in other words, in $N$, the set $S=j(S)\cap\kappa$ is a weakly compact subset of $\kappa$.
\end{proof}

\section{Weak compactness and forcing}

In what follows the elementary embedding characterization of weak compactness, i.e. $\Pi^1_1$-indescribability, is an essential ingredient, and here we review the required properties of this characterization. We say that $M$ is a \emph{$\kappa$-model} if and only if $M$ is a transitive model of $\ZFC^-$ such that $|M|=\kappa$, $\kappa\in M$ and $M^{<\kappa}\cap V\subseteq M$.

The following lemma, essentially due to Baumgartner \cite{MR0540770}, is well known. For details see one may consult \cite[Theorem 4.13]{MR2026390}, \cite[Theorem 6.4]{Kanamori:Book} or \cite[Theorem 16.1]{Cummings:Handbook}.

\begin{lemma}\label{lemma_characterizations_of_weak_compactness}
Given a set $S\in \mathscr{P}(\kappa)$, the following are equivalent.
\begin{enumerate}
\item $S$ is $\Pi^1_1$-indescribable.
\item For every $\kappa$-model $M$ with $\kappa,S\in M$ there is an elementary embedding $j:M\to N$ with critical point $\kappa$ such that $\kappa\in j(S)$ and $N$ is also a $\kappa$-model.
\item For any $\kappa$-model $M$ with $S\in M$, then there is a $\kappa$-complete nonprincipal $M$-ultrafilter $U$ on $\kappa$ such that the ultrapower $j_U:M\to N$ has critical point $\kappa$ and satisfies $\kappa\in j_U(S)$.
\item For all $A\subseteq\kappa$ there is a $\kappa$-model $M$ with $\kappa,A,S\in M$ and there is an elementary embedding $j:M\to N$ with critical point $\kappa$ such that $\kappa\in j(S)$ and $N$ is also a $\kappa$-model.
\end{enumerate}
\end{lemma}

\begin{remark}
If $j:M\to N$ is the ultrapower by an $M$-ultrafilter on $\kappa$ and $M^{<\kappa}\cap V\subseteq M$, then $N^{<\kappa}\cap V\subseteq N$.
\end{remark}

\begin{remark}\label{remark_preserving_weak_compactness}
Below we will need to prove that if a set $S\subseteq\kappa$ is weakly compact in the ground model $V$, then it remains so in a certain forcing extension $V[G]$, say where $G$ is $(V,\P)$-generic for some poset $\P$. To do this we will verify that (4) holds in $V[G]$. Specifically, we will fix $A\in\mathscr{P}(\kappa)^{V[G]}$ and argue that there is a $\P$-name $\dot{A}\in H_{\kappa^+}^V$ with $\dot{A}_G=A$. To find a $\kappa$-model and an embedding in $V[G]$, we proceed as follows. Applying the weak compactness of $S$ in $V$, let $M$ be a $\kappa$-model with $\kappa,S,\dot{A}\ldots\in M$ and let $j:M\to N$ be an elementary embedding with critical point $\kappa$ such that $\kappa\in j(S)$. Then we will argue that the embedding $j$ can be extended to $j:M[G]\to N[j(G)]$. Clearly, $A\in M[G]$ and it will follow that $M[G]$ is a $\kappa$-model in $V[G]$.
\end{remark}

The next lemma will be useful for construction of master conditions later on.

\begin{lemma}\label{lemma_X_not_weakly_compact_in_N}
Suppose $S\in \Pi^1_1(\kappa)^+$. Then for every $A\subseteq\kappa$ there is a $\kappa$-model $M$ with $\kappa,A,S\in M$ and an elementary embedding $j:M\to N$ with critical point $\kappa$, where $N$ is also a $\kappa$-model, such that $\kappa\in j(S)$ and $N\models$ $S\notin (\Pi^1_1(\kappa)^+)^N$.
\end{lemma}

\begin{proof}

Suppose $S\subseteq\kappa$ is weakly compact and fix $A\subseteq\kappa$. By Lemma \ref{lemma_characterizations_of_weak_compactness}, there is a $\kappa$-model $M$ with $\kappa,A,S\in M$ and there is an elementary embedding $j:M\to N$ with critical point $\kappa$ where $\kappa\in j(S)$ and $N$ is a $\kappa$-model. Choose such an embedding with $j(\kappa)$ as small as possible. We will show that $S$ is not weakly compact in $N$.

Suppose $S$ is weakly compact in $N$. Notice that $N$ may not satisfy the power set axiom, however since $V_\kappa\subseteq M$, it follows from the elementarity of $j$ that $N$ believes that there is a set $Y$ such that $X\in Y$ if and only if $X\subseteq\kappa$. We let $\mathscr{P}(\kappa)^N$ denote this set $Y$. Working in $N$, by Lemma \ref{lemma_characterizations_of_weak_compactness}, there is a $\kappa$-model $\bar{M}$ with $\kappa,A,S\in \bar{M}$ and there is an $N$-$\kappa$-complete nonprincipal $\bar{M}$-ultrafilter $F\subseteq \mathscr{P}(\kappa)^{N}$ such that the ultrapower embedding $i_N:\bar{M}\to (\bar{M})^\kappa/F=\{[h]_F\st h\in \bar{M}^\kappa\cap \bar{M}\}$ has critical point $\kappa$ and satisfies $\kappa \in i_N(S)$. Since $j(\kappa)$ is inaccessible in $N$, it follows that $i_N(\kappa)<j(\kappa)$. Since $F\in N$ is $N$-$\kappa$-complete and $N^{{<}\kappa}\cap V\subseteq N$, it follows that $F$ is $\kappa$-complete. Since $|\bar{M}|^N=\kappa$ we have $|\bar{M}|^V=\kappa$. Thus in $V$, $F$ is a $\kappa$-complete $\bar{M}$-ultrafilter and the ultrapower $i:\bar{M}\to (\bar{M})^\kappa/F$ is computed the same in $V$ as it is in $N$; that is, $i=i_N$. Hence $i(\kappa)<j(\kappa)$, which contradicts the minimality of $j(\kappa)$.
\end{proof}

In the proof of our main theorem we will use the following standard lemmas to argue that various instances of weak compactness are preserved in a particular forcing extension. For further discussion of these methods see \cite{Cummings:Handbook}.

\begin{lemma}\label{lemma_lifting_criterion}
Suppose $j:M\to N$ is an elementary embedding with critical point $\kappa$ where $M$ and $N$ are $\kappa$-models and $\P\in M$ is some forcing notion. Suppose $G\subseteq\P$ is a filter generic over $M$ and $H\subseteq j(\P)$ is a filter generic over $N$. Then $j$ extends to $j:M[G]\to N[H]$ if and only if $j[G]\subseteq H$.
\end{lemma}

Using terminology from \cite{Cummings:Handbook}, we say that a forcing $\P$ is \emph{$\kappa$-strategically closed} if and only if Player II has a winning strategy in the game $\mathcal{G}_\kappa(\P)$ of length $\kappa$ where Player II plays at even stages. For more information on strategic closure, see \cite[Section 5]{Cummings:Handbook}.

\begin{lemma}\label{lemma_diagonalization_criterion}
Suppose $M$ is a $\kappa$-model in $V$, so in particular $M^{<\kappa}\cap V\subseteq M$, and suppose $\P\in M$ is a forcing notion with $V\models$ ``$\P$ is $\kappa$-strategically closed''. Then there is a filter $G\in V$ generic for $\P$ over $M$.
\end{lemma}

\begin{lemma}\label{lemma_ground_closure}
Suppose that $M$ is a $\kappa$-model in $V$, so that in particular $M^{<\kappa}\cap V\subseteq M$, $\P\in M$ is some forcing notion and there is a filter $G\in V$ which is generic for $\P$ over $M$. Then $M[G]^{<\kappa}\cap V\subseteq M[G]$.
\end{lemma}

\begin{lemma}\label{lemma_generic_closure}
Suppose that $M$ is a $\kappa$-model in $V$, so in particular $M^{<\kappa}\cap V\subseteq M$, and suppose $\P\in M$ is $\kappa$-c.c. If $G\subseteq \P$ is generic over $V$, then $M[G]^{<\kappa}\cap V[G]\subseteq M[G]$.
\end{lemma}

In addition to arguing that a certain forcing iteration preserves instances of weak compactness, we will be concerned with showing that this iteration does not create new instances of weak compactness. Our main tool in this regard will be Lemma \ref{lemma_approximation_and_cover} below, which is due to Hamkins \cite{MR2063629}, but first we review some more basic results.

\begin{lemma}\label{lemma_ground_model_distributive}
A cardinal $\kappa$ is weakly compact after ${\leq}\kappa$-distributive forcing if and only if it was weakly compact in the ground model. Moreover, a set $S\subseteq\kappa$ is weakly compact after ${\leq}\kappa$-distributive forcing if and only if it was weakly compact in the ground model.
\end{lemma}

\begin{proof}
The weak compactness of a set $S\subseteq\kappa$ is witnessed by sets whose transitive closure has size at most $\kappa$, and since such objects are unaffected by ${\leq}\kappa$-distributive forcing, the result follows immediately.
\end{proof}

\begin{lemma}\label{lemma_small_forcing}
Suppose $\kappa$ is weakly compact and $\P$ is a forcing of size less than $\kappa$. Then $\kappa$ is weakly compact after forcing with $\P$ if and only if it was weakly compact in the ground model.
\end{lemma}

\begin{lemma}\label{lemma_ground_model_strategically}
If $\kappa$ is weakly compact after $\kappa$-strategically closed forcing, then it was weakly compact in the ground model.
\end{lemma}

\begin{proof}
Suppose $\P$ is $\kappa$-strategically closed and $G\subseteq\P$ is generic over $V$. Any $\kappa$-tree $T$ in $V$ remains a $\kappa$-tree in $V[G]$, and thus $T$ has a cofinal branch in $V[G]$. Using a name for this branch, and the strategic closure of $\P$, one may construct a cofinal branch through $T$ in $V$.
\end{proof}

Let us recall two definitions from \cite{MR2063629}. A pair of transitive classes $M\subseteq N$ satisfies the \emph{$\delta$-approximation property} if whenever $A\subseteq M$ is a set in $N$ and $A\cap a\in M$ for any $a\in M$ of size less than $\delta$ in $M$, then $A\in M$. The pair $M\subseteq N$ satisfies the \emph{$\delta$-cover property} if for every set $A$ in $N$ with $A\subseteq M$ and $|A|^N<\delta$, there is a set $B\in M$ with $A\subseteq B$ and $|B|^M<\delta$. One can easily see that many Easton-support iterations $\P$ of length greater than a Mahlo cardinal $\delta$ satisfy the $\delta$-approximation and cover properties by factoring the iteration as $\P\cong\Q*\dot{\R}$ where $\Q$ is $\delta$-c.c. and $\dot{\R}$ is forced by $\Q$ to be ${<}\delta$-strategically closed.

\begin{lemma}[Hamkins, \cite{MR2063629}]\label{lemma_approximation_and_cover}
Suppose that $\kappa$ is a weakly compact cardinal, $S\in P(\kappa)^V$ and $V\subseteq\overline{V}$ satisfies the $\delta$-approximation and cover properties for some $\delta<\kappa$. If $S$ is a weakly compact subset of $\kappa$ in $\overline{V}$ then $S$ is a weakly compact subset of $\kappa$ in $V$.
\end{lemma}

\begin{proof}
Suppose $S\in P(\kappa)^V$ is a weakly compact subset of $\kappa$ in $\overline{V}$. Fix $A\in P(\kappa)^V$. By \cite[Lemma 16]{MR2063629}, there is a transitive model $\overline{M}\in\overline{V}$ of some large fixed finite fragment $\ZFC^*$ of $\ZFC$ with $|\overline{M}|^{\overline{V}}=\kappa$ such that $\kappa,A,S\in\overline{M}$, the model $\overline{M}$ is closed under ${<}\kappa$-sequences from $\overline{V}$ and $M=\overline{M}\cap V\in V$ is a transitive model of the finite fragment $\ZFC^*$ with $|M|^V=\kappa$. Since $S$ is weakly compact in $\overline{V}$, it follows that there is an elementary embedding $j:\overline{M}\to\overline{N}$ where $\overline{N}^{<\kappa}\cap\overline{V}\subseteq\overline{N}$ and $\kappa\in j(S)$. Since this embedding satisfies the hypotheses of the main theorem from \cite{MR2063629}, it follows that $j\restrict M:M\to N$ is an elementary embedding in $V$ with critical point $\kappa$. Since $A,S\in M$ and $\kappa\in (j\restrict M)(S)$ we see that $S$ is a weakly compact subset of $\kappa$ in $V$.
\end{proof}



\section{Adding a nonreflecting weakly compact set}

In Lemma \ref{lemma_weakly_compacts_are_weakly_compact} above we proved that $\Refl_n(\kappa)$ implies that $\kappa$ is $\omega$-$\Pi^1_n$-indescribable. Taking $n=1$, this shows that if the weakly compact reflection principle holds at $\kappa$ then $\kappa$ is $\omega$-weakly compact. We now show that the converse is consistently false. Let $\WC$ denote the class of weakly compact cardinals and let $\WC_\kappa=\WC\cap\kappa$. Indeed we will prove that if $\kappa$ is $(\alpha+1)$-weakly compact, then there is an Easton-support forcing iteration $\P_{\kappa+1}$ of length $\kappa+1$ such that in $V^{\P_{\kappa+1}}$ we have (1) there is a nonreflecting weakly compact subset of $\kappa$, (2) $\WC_\kappa^V=\WC_\kappa^{V^{\P_{\kappa+1}}}$ and (3) $\kappa$ remains $(\alpha+1)$-weakly compact.

For a cardinal $\gamma$ and a cofinal subset $W\subseteq\gamma$, we define a forcing notion $\Q(\gamma,W)$ as follows. Let $p$ be a condition in $\Q(\gamma,W)$ if and only if
\begin{enumerate}
\item $p$ is a function with $\dom(p)<\gamma$ and $\range(p)\subseteq 2$,
\item if $\eta\leq\dom(p)$ is weakly compact then $\{\alpha<\eta\st p(\alpha)=1\}$ is not a weakly compact subset of $\eta$ and
\item $\supp(p)\subseteq W$.
\end{enumerate} For $p,q\in\Q(\gamma,W)$ let $p\leq q$ if and only if $p\supseteq q$. 



\begin{lemma}\label{lemma_extending}
Assuming that the collection of weakly compact limit points of $W$ is cofinal in $\gamma$, every condition $q\in\Q(\gamma,W)$ can be extended to a condition $p\leq q$ such that $\dom(p)$ is a weakly compact cardinal and $\supp(p)$ is cofinal in $\dom(p)$.
\end{lemma}

\begin{proof}
Let $\delta$ be the least element of $\WC_\gamma\cap \Lim(W)$ greater than $\dom(q)$. One can use the usual reflection arguments to show that $\WC_\delta\cap W$ is not a weakly compact subset of $\delta$ (if it were then $\delta$ would not be the least such cardinal by elementarity). Define $p:\delta\to 2$ by letting $p\restrict[\dom(q),\delta)$ be the characteristic function of $\WC_\delta\cap W\cap[\dom(q),\delta)$ and $p\restrict\dom(q)=q$.
\end{proof}

The proof of the next lemma is very similar to that of the analogous fact about the forcing to add a nonreflecting stationary set of cofinality $\omega$ ordinals in a regular cardinal \cite[Example 6.5]{Cummings:Handbook}.

\begin{lemma}
Suppose $\mu\leq\gamma$ is the least weakly compact cardinal such that $W\cap\mu$ is a weakly compact subset of $\mu$. Then $\Q(\gamma,W)$ is ${<}\mu$-closed and $\gamma$-strategically closed, but not ${\leq}\mu$-closed.
\end{lemma}

\begin{proof}
It is easy to see that $\Q(\gamma,W)$ is ${<}\mu$-closed. Suppose $\delta<\mu$ and $\<p_i\st i<\delta\>$ is a decreasing sequence of conditions in $\Q(\gamma,W)$. Let $\eta=\sup\{\dom(p_i)\st i<\delta\}$. If $\eta$ is not weakly compact then $\bigcup\{p_i\st i<\delta\}\in \Q(\gamma,W)$ is a lower bound of the sequence. On the other hand, if $\eta$ is weakly compact, then $\eta=\delta<\mu$, which implies that $W\cap\eta$ is not a weakly compact subset of $\eta$. It follows that the support $\supp(\bigcup\{p_i\st i<\delta\})$ is a subset of $W$ and hence is not a weakly compact subset of $\eta$. Thus $\bigcup\{p_i\st i<\delta\}$ is a condition in $\Q(\gamma,W)$.

It is also easy to see that $\Q(\gamma,W)$ is not ${\leq}\mu$-closed. For each $i<\mu$ let $p_i$ be the characteristic function of $W\cap i$. Then $\<p_i\st i<\mu\>$ is a decreasing sequence of conditions in $\Q(\gamma,W)$ with no lower bound.

To prove that $\Q(\gamma,W)$ is $\gamma$-strategically closed we must argue that Player II has a winning strategy in the game $\mathcal{G}_\gamma(\Q(\gamma,W))$. The game begins with Player II playing $p_0=\emptyset$. As the game proceeds, Player II may use the following strategy. At an even stage $\alpha$, Player II calculates $\gamma_\alpha=\sup\{\dom(p_i)\st i<\alpha\}$ and then defines $p_\alpha$ by setting $\dom(p_\alpha)=\gamma_\alpha+1$, $p_\alpha\restrict\gamma_\alpha=\bigcup_{i<\alpha}p_i$ and $p_\alpha(\gamma_\alpha)=0$. To check that this strategy succeeds, we must verify that for all limit stages $\eta<\gamma$ the strategy produces a condition $p_\eta$. If $\eta<\gamma$ is not weakly compact then $p_\eta$ is clearly a condition. If $\eta<\gamma$ is weakly compact one of two things must occur. Either $\gamma_\eta=\sup\{\gamma_i\st i<\eta\}$ is equal to $\eta$ or $\gamma_\eta>\eta$. If $\gamma_\eta>\eta$ then $\gamma_\eta$ is singular, in which case $p_\eta=\bigcup_{i<\eta}p_i\cup\{(\gamma_\eta,0)\}$ is a condition. Otherwise, if $\gamma_\eta=\eta$, since $\gamma_\eta$ is weakly compact, in order to see that $p_\eta=\bigcup_{i<\eta}p_i\cup\{(\gamma_\eta,0)\}$ is a condition, we must check that $\{\alpha<\eta\st p_\eta(\alpha)=1\}$ is not weakly compact. The set $\{\gamma_i\st i<\eta\}$ is club in $\eta=\gamma_\eta$ and Player II has ensured that $p_\eta(\gamma_i)=0$ for all $i<\eta$. Thus $\{\alpha<\eta\st p_\eta(\alpha)=1\}$ is nonstationary and hence not weakly compact.
\end{proof}

\begin{lemma}
For every cardinal $\delta<\gamma$, there is a ${<}\delta$-directed closed open dense subset of $\Q(\gamma,W)$.
\end{lemma}

\begin{proof}
Suppose $\delta<\gamma$. Let $D=\{p\in\Q(\gamma,W)\st \dom(p)>\delta\}$. Clearly $D$ is open and dense. Suppose $A\subseteq D$ is a directed set of conditions and $|A|<\delta$. Let $\eta=\sup\{\dom(p)\st p\in A\}$. Since $\eta>\delta$ and $\cf(\eta)\leq |A|<\delta$, it follows that $\eta$ is not weakly compact. Thus $p=\bigcup A$ is a condition in $\Q(\gamma,W)$.
\end{proof}

Let $W\subseteq\WC_\kappa$ be a weakly compact subset of $\kappa$ and suppose $\Refl_{wc}(\kappa)$ holds. Let $\P_{\kappa+1}=\<(\P_\gamma,\dot{\Q}_\gamma)\st \gamma\leq\kappa \>$ be the length $\kappa+1$ iteration with Easton-support such that the stage $\gamma$ forcing is defined as follows.
\begin{itemize}
\item If $\gamma\leq\kappa$ is a Mahlo limit point of $\Tr_\wc(W)$, then $\dot{\Q}_\gamma$ is a $\P_\gamma$-name for the forcing $\Q(\gamma,W\cap\gamma)$ defined above.\footnote{It is important that we force at some stages $\gamma$ with $\Q(\gamma,W)$ at which $W\cap\gamma$ is not a weakly compact subset of $\gamma$ because in a crucial part of the argument we will have that on the $j$-side, where $j$ is some elementary embedding with critical point $\kappa$, the set $W=j(W)\cap\kappa$ is not a weakly compact subset of $\kappa$ and this is required in order for a master condition to exist.}
\item Otherwise, $\dot{\Q}_\gamma$ is a $\P_\gamma$-name for trivial forcing.
\end{itemize}

Next we will show that, as intended, in certain contexts, the iteration $\P_{\kappa+1}$ adds a nonreflecting weakly compact set and preserves many instances of weak compactness. In the next theorem, the assumption that $\Refl_\wc(\kappa)$ holds is made to avoid trivialities. For example: if $\Refl_\wc(\kappa)$ fails, then $\kappa$ already has a nonreflecting weakly compact subset.

\begin{theorem}\label{theorem_adding_a_nonreflectin_weakly_compact_set}
Suppose $\Refl_\wc(\kappa)$ holds, $W\subseteq\WC_\kappa$ is a weakly compact subset of $\kappa$ and $\GCH$ holds. Then there is a cofinality-preserving forcing $\P$ such that if $G$ is $(V,\P)$-generic then the following conditions hold.
\begin{enumerate}
\item In $V[G]$, there is a nonreflecting weakly compact set $E\subseteq W$ (thus $W$ remains weakly compact).
\item The class of weakly compact cardinals is preserved: $\WC^V=\WC^{V[G]}$.
\item If $S\subseteq\gamma<\kappa$ is a weakly compact subset of $\gamma$ in $V$ and $W\cap\gamma\subseteq S$, then $S$ remains weakly compact in $V[G]$.
\end{enumerate}
\end{theorem}

\begin{proof}

Let $G_{\kappa+1}$ be $(V,\P_{\kappa+1})$-generic. Since the iteration $\P_{\kappa+1}$ uses Easton-support and at each nontrivial stage $\gamma\leq\kappa$ we have $\forced_{\P_\gamma}$ ``$\dot{\Q}_\gamma$ is $\gamma$-strategically closed'', standard arguments involving factoring the iteration show that cofinalities are preserved.

The forcing $\P_{\kappa+1}$ can be factored as $\P_\kappa*\dot{\Q}(\kappa,W)$. For $\gamma\leq\kappa$, let $H_\gamma$ denote the $(V[G_\gamma],\Q_\gamma)$-generic filter obtained from $G_{\kappa+1}$. Thus, $G_{\kappa+1} = G_\kappa*H_\kappa$. Let $f=\bigcup H_\kappa:\kappa\to 2$ and let $E=\{\alpha<\kappa\st f(\alpha)=1\}$.

First we prove (1). Let us show that in $V[G_{\kappa+1}]$, for each $\eta<\kappa$, the set $E\cap\eta$ is not a weakly compact subset of $\eta$. If $\eta$ is not weakly compact in $V[G_{\kappa+1}]$ then neither is $E\cap\eta$. Suppose $\eta$ is weakly compact in $V[G_{\kappa+1}]$. Since $\Q(\kappa,W)$ is $\kappa$-strategically closed, it follows that $E\cap\eta\in V[G_\kappa]$ and thus $f\restrict\eta$ is a condition in $\Q(\kappa,W)$. By Lemma \ref{lemma_ground_model_distributive},  $\eta$ is weakly compact in $V[G_\kappa]$. Hence by definition of $\Q(\kappa,W)$, the set $E\cap\eta$ is not weakly compact in $V[G_\kappa]$. Applying Lemma \ref{lemma_ground_model_distributive} again, we conclude that $E\cap \eta$ is not weakly compact in $V[G_{\kappa+1}]$.

We will show that in the extension $V[G_{\kappa+1}]$, the set $E\subseteq\kappa$ is a weakly compact subset of $\kappa$ using the method outlined in Remark \ref{remark_preserving_weak_compactness} above. Fix $A\in \mathscr{P}(\kappa)^{V[G_{\kappa+1}]}$ and let $\dot{E},\dot{A}\in H_{\kappa^+}$ be $\P_{\kappa+1}$-names with $E=\dot{E}_{G_{\kappa+1}}$ and $A=\dot{A}_{G_{\kappa+1}}$. Working in $V$, we may apply Lemma \ref{lemma_X_not_weakly_compact_in_N} to find a $\kappa$-model $M$ with $\kappa,\P_{\kappa+1},\dot{A},\dot{E},W\in M$ and an elementary embedding $j:M\to N$ with critical point $\kappa$ where $N$ is also a $\kappa$-model such that $\kappa\in j(W)$ and $W$ is not a weakly compact subset of $\kappa$ in $N$. Since $\kappa$ is an inaccessible limit point of $\Tr_{\wc}(W)^N=\Tr_\wc(W)^V$ in $N$ and $N^{<\kappa}\cap V\subseteq N$ we have $j(\P_\kappa)\cong \P_\kappa*\dot{\Q}(\kappa,W)*\dot{\P}'_{\kappa,j(\kappa)}$ where $\dot{\P}'_{\kappa,j(\kappa)}$ is a $\P_{\kappa+1}$-name for the tail of $N$'s version of the iteration $j(\P_\kappa)$. Since $\P_\kappa$ is $\kappa$-c.c. in $V$ and $\Q(\kappa,W)$ is ${<}\kappa$-distributive in $V[G_\kappa]$, it follows from Lemma \ref{lemma_generic_closure}, that $M[G_{\kappa+1}]$ is closed under ${<}\kappa$-sequences in $V[G_{\kappa+1}]$. Thus, using the facts that $|N|^V=\kappa$ and that in $N[G_{\kappa+1}]$, the poset $\P_{\kappa,j(\kappa)}$ contains a ${<}\kappa$-directed closed dense subset, we can build an $(N[G_{\kappa+1}]$, $\P_{\kappa,j(\kappa)})$-generic filter $K\in V[G_{\kappa+1}]$. Since the supports of conditions in $G$ have size less than $\kappa$ we have $j[G]\subseteq G_\kappa*H_\kappa*K$, and thus we can extend the embedding to $j:M[G_\kappa]\to N[\widehat{G}_{j(\kappa)}]$ where $\widehat{G}_{j(\kappa)}=G_\kappa*H_\kappa*K$. Since $K\in V[G_{\kappa+1}]$ we have $N[\widehat{G}_{j(\kappa)}]^{<\kappa}\cap V[G_{\kappa+1}]\subseteq N[\widehat{G}_{j(\kappa)}]$. Since $W$ is not weakly compact in $N$, it follows from Lemma \ref{lemma_approximation_and_cover} that $W$ is not weakly compact in $N[\widehat{G}_{j(\kappa)}]$.  

Let us now argue that $f\cup\{(\kappa,1)\}$ is a condition in $j(\Q(\kappa,W))$. We have $f,E\in N[\widehat{G}_{j(\kappa)}]$ because $H_\kappa\in N[\widehat{G}_{j(\kappa)}]$. It will suffice to show that the set $E$ is not a weakly compact subset of $\kappa$ in $N[\widehat{G}_{j(\kappa)}]$. From the definition of $\Q(\kappa,W)\in V[G_\kappa]$, we have $E\subseteq W$, and since $W$ is not a weakly compact subset of $\kappa$ in $N[\widehat{G}_{j(\kappa)}]$, it follows that $E$ is not a weakly compact subset of $\kappa$ in $N[\widehat{G}_{j(\kappa)}]$. This implies that $\bigcup j[H_\kappa]=\bigcup H_\kappa = f$ is a condition in $j(\Q(\kappa,W))$, and thus so is $p=f\cup\{(\kappa,1)\}$.

In $N[\widehat{G}_{j(\kappa)}]$, the poset $j(\Q(\kappa,W))$ contains a ${\leq}\kappa$-closed dense subset, and hence we can build an $(N[\widehat{G}]_{j(\kappa)},j(\Q(\kappa,W)))$-generic filter $\widehat{H}_{j(\kappa)}\in V[G_{\kappa+1}]$ with $p\in \widehat{H}_{j(\kappa)}$. This guarantees that $j[H_\kappa]\subseteq\widehat{H}_{j(\kappa)}$, and thus the embedding extends to $j:M[G_\kappa*H_\kappa]\to N[\widehat{G}_{j(\kappa)}*\widehat{H}_{j(\kappa)}]$. Since $(\kappa,1)\in \bigcup \widehat{H}_{j(\kappa)}$, it follows that $\kappa\in j(E)$. Therefore $E$ is a weakly compact subset of $\kappa$ in $V[G*g]$, and hence we have established (1).

For (2), let us now show that the class of weakly compact cardinals is preserved; i.e. that $\WC^{V[G_{\kappa+1}]}=\WC^V$. By the work of Hamkins, it follows that $\WC^{V[G_{\kappa+1}]}\subseteq \WC^V$ (see Lemma \ref{lemma_approximation_and_cover} above or \cite[Lemma 16]{MR2063629}). We must show that $\WC^{V[G_{\kappa+1}]}\supseteq \WC^V$. Suppose $\gamma$ is weakly compact in $V$. If $\gamma>\kappa$, then by Lemma \ref{lemma_small_forcing}, it follows that $\gamma$ is weakly compact in $V[G_{\kappa+1}]$. Suppose $\gamma<\kappa$. Since in $V[G_{\gamma+1}]$ the forcing $\P_{\gamma,\kappa+1}$ is ${\leq}\gamma$-distributive, it will suffice to argue that $\gamma$ remains weakly compact in $V[G_{\gamma+1}]$. There are several cases to consider.

\textsc{Case 1.} Suppose $\gamma$ is a limit point of $\Tr_\wc(W)$. Then $\P_{\gamma+1}\cong\P_\gamma*\dot{\Q}(\gamma,W\cap\gamma)$. Working in $V[G_\gamma*H_\gamma]$, fix $A\subseteq\gamma$ and let $\dot{A}\in H_{\kappa^+}$ be a $\P_\gamma*\dot{\Q}(\gamma,W\cap\gamma)$-name such that $\dot{A}_{G_\gamma*H_\gamma}=A$. Working in $V$, by Lemma \ref{lemma_X_not_weakly_compact_in_N}, it follows that there is a $\gamma$-model $M$ with $\gamma,\dot{A},\P_{\gamma+1},W\cap\gamma\in M$ and an elementary embedding $j:M\to N$ with critical point $\gamma$ such that $W\cap \gamma$ is not weakly compact in $N$.\footnote{Of course $W\cap\gamma$ may or may not be a weakly compact subset of $\gamma$ in $V$, but in either case such an embedding exists. If $W\cap \gamma$ is a weakly compact subset of $\gamma$, then Lemma \ref{lemma_X_not_weakly_compact_in_N} applies. If $W\cap\gamma$ is not a weakly compact subset of $\gamma$, let $C\subseteq\gamma$ be a $1$-club with $C\cap W\neq\emptyset$, let $M$ be a $\gamma$-model with $\gamma,\dot{A},\P_{\gamma+1},W\cap\gamma,C\in M$ and let $j:M\to N$ be an elementary embedding with critical point $\gamma$, where $N$ is also a $\gamma$-model. Then since $C\in N$ and $C$ is $1$-club in $N$, it follows that $W\cap\gamma$ is not weakly compact in $N$.} Since $\gamma$ is a \emph{Mahlo} limit point of $\Tr_\wc(W)$ in $N$, it follows that $j(\P_\gamma)\cong\P_\gamma*\dot{\Q}(\gamma,W\cap\gamma)*\dot{\P}'_{\gamma,j(\gamma)}$ where $\dot{\P}'_{\gamma,j(\gamma)}$ is a $\P_{\gamma+1}$-name for $N$'s version of the iteration from $\gamma+1$ to $j(\gamma)$. It follows that $G_\gamma$ is $(M,\P_\gamma)$-generic and $G_\gamma*H_\gamma$ is $(N,\P_\gamma*\dot{\Q}(\gamma,W\cap\gamma))$-generic, and furthermore that $N[G_\gamma*H_\gamma]^{<\gamma}\cap V[G_\gamma*H_\gamma]\subseteq N[G_\gamma*H_\gamma]$. Thus, as before, we may build an $(N[G_\gamma*H_\gamma],\P_{\gamma,j(\gamma)})$-generic filter $K\in V[G_\gamma*H_\gamma]$. Let $\widehat{G}_{j(\gamma)}=G_\gamma*H_\gamma*K$ Since $j[G_\gamma]\subseteq \widehat{G}_{j(\gamma)}$ the embedding extends to $j:M[G_\gamma]\to N[\widehat{G}_{j(\gamma)}]$. Let $m=\bigcup H_\gamma$. From the definition of $\dot{\Q}(\gamma,W\cap\gamma)$, it follows that $\supp(m)\subseteq W\cap\gamma$. Furthermore, since $W\cap\gamma$ is not weakly compact in $N$, it follows by Lemma \ref{lemma_approximation_and_cover}, that $W\cap\gamma$ is not weakly compact in $N[\widehat{G}_{j(\gamma)}]$. Thus, $m\in j(\Q(\gamma,W\cap\gamma))$. We can build an $(N[\widehat{G}_{j(\gamma)}],j(\Q(\gamma,W\cap\gamma)))$-generic filter $\widehat{H}_{j(\gamma)}\in V[G_\gamma*H_\gamma]$ with $m\in\widehat{G}_{j(\gamma)}$. The embedding extends to $j:M[G_\gamma*H_\gamma]\to N[\widehat{G}_{j(\gamma)}*\widehat{H}_{j(\gamma)}]$, and thus $\gamma$ is weakly compact in $V[G_{\gamma+1}]$.

\textsc{Case 2.} Suppose $\gamma$ is not a limit point of $\Tr_\wc(W)$. In this case $\P_{\gamma+1}\cong \P_\gamma$. If $j:M\to N$ witnesses the weak compactness of $\gamma$ in $V$ with $\P_{\gamma+1},\Tr_{\wc}(W),\ldots\in M$ then $\gamma$ is not a limit point of $\Tr_{\wc}(W)$ in $N$ and hence $j(\P_{\gamma+1})\cong \P_\gamma * \dot{\P}_{\gamma,j(\gamma)}'$ where $\dot{\P}_{\gamma,j(\gamma)}$ is a $\P_\gamma$-name for $N$'s version of the iteration from $\gamma+1$ to $j(\gamma)$. One may now use a standard argument as in Case I to show that $\gamma$ is weakly compact in $V[G_{\gamma+1}]$.

Thus the class of weakly compact cardinals is preserved $\WC^{V[G_{\kappa+1}]}=\WC^V$, establishing (2).

To prove (3), suppose $S\subseteq\gamma<\kappa$ is a weakly compact subset of $\gamma$ in $V$ and $W\cap\gamma\subseteq S$. To show that $S$ is weakly compact in $V[G_{\kappa+1}]$ it will suffice, by Lemma \ref{lemma_ground_model_distributive}, to argue that $S$ is weakly compact in $V[G_{\gamma+1}]$. Fix $A\in\mathscr{P}(\kappa)^{V[G_{\gamma+1}]}$ and let $\dot{A}\in H_{\gamma^+}$ be a $\P_{\gamma+1}$-name with $\dot{A}_G=A$. By Lemma \ref{lemma_X_not_weakly_compact_in_N}, we may let $M$ be a $\gamma$-model with $\gamma,\dot{A},\P_{\gamma+1},S,W\cap\gamma\in M$ and let $j:M\to N$ be an elementary embedding with critical point $\gamma$ where $N$ is a $\gamma$-model such that $\kappa\in j(S)$ and $S=j(S)\cap\gamma$ is not a weakly compact subset of $\gamma$ in $N$. Since $W\cap\gamma\subseteq S$, it follows that $W$ is not a weakly compact subset of $\gamma$ in $N$, and thus we can lift the embedding using a master condition argument as in \textsc{Case 1} or \textsc{Case 2} above. Thus $S$ remains a weakly compact subset of $\gamma$ in $V[G_{\gamma+1}]$.
\end{proof}

We now show that Theorem \ref{theorem_main_theorem} is a consequence of Theorem \ref{theorem_adding_a_nonreflectin_weakly_compact_set}, which establishes the consistency of the failure of the weakly compact reflection principle $\Refl_\wc(\kappa)$ at a weakly compact cardinal of any order $\gamma<\kappa^+$.

\begin{proof}[Proof of Theorem \ref{theorem_main_theorem}]
Suppose $\kappa$ is $(\alpha+1)$-weakly compact where $\omega\leq\alpha<\kappa$. We must show that there is a cofinality-preserving forcing extension in which there is a nonreflecting weakly compact subset of $\kappa$, the class of weakly compact cardinals is preserved and $\kappa$ remains $(\alpha+1)$-weakly compact. 

Choose a sequence $\vec{A}=\<A_\xi\st\xi<\kappa^+\>$ of subset of $\kappa$ with $A_0=\kappa$ such that for all $\xi<\kappa^+$,
\begin{enumerate}
\item $A_\xi$ is a representative of the equivalence class $\Tr_\wc^\xi([\kappa]_1)$,
\item $A_{\xi+1}=\Tr_\wc(A_\xi)$ and
\item if $\xi$ is a limit then $A_\xi=\bigtriangleup\{A_\zeta\st\zeta<\xi\}$.
\end{enumerate}  
Let $W=A_\alpha\in\Tr_\wc^\alpha([\kappa]_1)$. Let $\P_{\kappa+1}$ be the forcing from the proof of Theorem \ref{theorem_adding_a_nonreflectin_weakly_compact_set}, and suppose $G_{\kappa+1}$ is $(V,\P_{\kappa+1})$-generic. Applying Theorem \ref{theorem_adding_a_nonreflectin_weakly_compact_set} (1), we conclude that in $V[G_{\kappa+1}]$, the set $W$ is weakly compact and the set $E\subseteq W$ is a nonreflecting weakly compact subset of $\kappa$. By Theorem \ref{theorem_adding_a_nonreflectin_weakly_compact_set} (2), the class of weakly compact cardinals is preserved. It remains to show that $\kappa$ remains $(\alpha+1)$-weakly compact in $V[G_{\kappa+1}]$.

Working in $V[G_{\kappa+1}]$, let $\vec{B}=\<B_\xi\st\xi<\kappa^+\>$ be a sequence defined using the same conditions that were used to define $\vec{A}$, but this time we run the definition in $V[G_{\kappa+1}]$; that is, $B_0=\kappa$ and for all $\xi<\kappa^+$,
\begin{enumerate}
\item $B_\xi$ is a representative of the equivalence class $\Tr_\wc^\xi([\kappa]_1)^{V[G_{\kappa+1}]}$,
\item $A_{\xi+1}=\Tr_\wc(A_\xi)^{V[G_{\kappa+1}]}$ and
\item if $\xi$ is a limit then $A_\xi=\bigtriangleup\{A_\zeta\st\zeta<\xi\}^{V[G_{\kappa+1}]}$.
\end{enumerate}
To show that $\kappa$ is $(\alpha+1)$-weakly compact in $V[G_{\kappa+1}]$, it will suffice to show that $W=A_\alpha\subseteq B_\alpha$. We will use induction to prove that for every $\xi\leq\alpha$ we have $A_\xi\subseteq B_\xi$. 

By definition of the sequences we have $A_0=\kappa=B_0$. If $\xi\leq\alpha$ is a limit the result follows immediately from the inductive hypothesis that $A_\zeta\subseteq B_\zeta$ for all $\zeta<\xi$. If $\xi=\zeta+1\leq\alpha$ is a successor, then $A_\xi=A_{\zeta+1}=\Tr_\wc(A_\zeta)$ and $B_\xi=B_{\zeta+1}=\Tr_\wc(B_\zeta)$. Let us show that $A_\xi\subseteq B_\xi$. Suppose $\gamma\in A_\xi$, this means that $A_\zeta\cap\gamma$ is a weakly compact subset of $\gamma$ in $V$. By Theorem \ref{theorem_adding_a_nonreflectin_weakly_compact_set} (3), it follows that $A_\zeta\cap\gamma$ remains a weakly compact subset of $\gamma$ in $V[G_{\kappa+1}]$. By the inductive hypothesis, $A_\zeta\subseteq B_\zeta$ and thus $B_\zeta\cap\gamma$ is a weakly compact subset of $\gamma$ in $V[G_{\kappa+1}]$. Thus we have shown that $\Tr_\wc^\alpha([\kappa]_1)=[B_\alpha]_1>0$ in $V[G_{\kappa+1}]$ and hence $\kappa$ remains $(\alpha+1)$-weakly compact in $V[G_{\kappa+1}]$.
\end{proof}

\section{Questions}\label{section_questions}

Many questions regarding the weakly compact and $\Pi^1_n$-reflection principles remain open.

\subsection{Consistency strength}

\begin{question}
What is the consistency strength of the statement ``there is a cardinal $\kappa$ such that the weakly compact reflection principle $\Refl_\wc(\kappa)$ holds''? More generally, what is the strength of ``there is a cardinal $\kappa$ such that the $\Pi^1_n$-reflection principle $\Refl_n(\kappa)$ holds''?
\end{question}

Mekler and Shelah \cite{MR1029909} showed that the statement ``there is a regular cardinal $\kappa$ such that every stationary subset of $\kappa$ has a stationary initial segment'' is equiconsistent with the existence of a \emph{reflection cardinal}; i.e. a cardinal $\kappa$ which carries a normal ideal $I$ coherent with the nonstationary ideal such that the $I$-positive sets are closed under the operation $\Tr_0$ where $\Tr_0(X)=\{\alpha<\kappa\st \text{$X\cap\alpha$ is stationary}\}$. As shown in \cite{MR1029909}, the existence of a reflection cardinal is a hypothesis who's consistency strength is strictly between that of a greatly Mahlo cardinal and a weakly compact cardinal. 
\begin{definition}
A cardinal $\kappa$ is a \emph{weak compactness reflection cardinal} if and only if there exists a normal ideal $I$ on $\kappa$ such that 
\begin{enumerate}
\item $I$ is coherent to the weakly compact ideal $\Pi^1_1(\kappa)$ and
\item the $I$-positive sets $I^+$ are closed under the operation $\Tr_\wc$.
\end{enumerate}
\end{definition}

We conjecture that the existence of a cardinal at which the weakly compact reflection principle holds is equiconisistent with the existence of a weak compactness reflection cardinal, and this hypothesis is strictly between the existence of a greatly weakly compact cardinal and the existence of a $\Pi^1_2$-indescribable cardinal in the large cardinal hierarchy. Notice that the forward direction of the above conjectured equiconsistency easily follows from Lemma \ref{lemma_trace_of_wc_is_wc}: if $\Refl_{wc}(\kappa)$ holds then $\kappa$ is a weak compactness reflection cardinal since $\Pi^1_1(\kappa)$ is a normal ideal and $\Pi^1_1(\kappa)^+$ is closed under the operation $\Tr_\wc$. One may be able to establish the remaining parts of the conjecture, by carrying out arguments similar to \cite[Corollary 5 and Theorem 7]{MR1029909}, however many technical difficulties need to be overcome. Similarly, one can define the notion of $\Pi^1_n$-reflection cardinal and formulate a similar conjecture regarding the consistency strength of the $\Pi^1_n$-reflection principle.

\subsection{Great weak compactness}

Above we have established that $\Refl_\wc(\kappa)$ implies that $\kappa$ is $\omega$-weakly compact, and that it is consistent that $\Refl_\wc(\kappa)$ fails and $\kappa$ is $\gamma$-weakly compact, for any fixed $\gamma<\kappa^+$.
\begin{question}\label{question_greatly}
Is it consistent that there exists a greatly weakly compact cardinal $\kappa$ such that $\Refl_\wc(\kappa)$ fails?
\end{question}
Regarding Question \ref{question_greatly}, in our proof of Theorem \ref{theorem_main_theorem}, we were able to show that the weak compactness of a set $W=A_\alpha$ is preserved by the iteration $\P_{\kappa+1}$ as follows. We showed that the set $E\subseteq W$ added by the iteration $\P_{\kappa+1}$ is weakly compact, by choosing an embedding $j:M\to N$ with critical point $\kappa$ such that $\kappa\in j(W)$ and $W$ is not weakly compact in $N$. This allowed us to conclude that $W$ remains weakly compact, since it contains $E$. In order to preserve the great weak compactness of $\kappa$, one would want to argue that the weak compactness of $\kappa^+$ subsets of $\kappa$ is preserved. It is not clear that the techniques used in this article could be adapted to accomplish this.

\subsection{The least $\omega$-weakly compact cardinal}

One may also be able to adapt another argument \cite[Theorem 9]{MR1029909} of Mekler and Shelah to answer the following question.
\begin{question}
Is it consistent that $\Refl_\wc(\kappa)$ holds and $\kappa$ is not $(\omega+1)$-weakly compact? Is it consistent that $\Refl_\wc(\kappa)$ holds at the least $\omega$-weakly compact cardinal?\footnote{This question has been answered in the affirmative; see the forthcoming paper by the author and Hiroshi Sakai.}
\end{question}

\subsection{Resurrection of the weakly compact reflection prinicple}\label{section_resurrecting}


 Recall that if $\kappa$ is weakly compact, then there is a forcing $\P$ such that in $V^\P$ the weak compactness of $\kappa$ is indestructible by $\Add(\kappa,1)$.\footnote{Suppose $\kappa$ is weakly compact and let $\P$ be the standard Easton-support iteration of length $\kappa$ which adds a Cohen subset to each inaccessible cardinal less than $\kappa$. Then standard arguments show that, in $V^\P$, $\kappa$ remains weakly compact in $V^\P$ and the weak compactness of $\kappa$ is indestructible by the forcing to add a Cohen subset of $\kappa$.} Working in $V^\P$, let $\mathbb{S}$ be the natural forcing for adding a nonreflecting stationary subset of $\kappa$ and let $\dot{S}$ be an $\mathbb{S}$-name for this set. Then $\kappa$ is not weakly compact in $V^{\P*\dot{\mathbb{S}}}$. Working in $V^{\P*\dot{\mathbb{S}}}$, let $\mathbb{C}$ be the forcing to shoot a club through $\kappa\setminus \dot{S}$. One can argue that, in $V^\P$, the separative poset $\mathbb{S}*\dot{\mathbb{C}}$ contains a ${<}\kappa$-closed dense subset of size $\kappa$ and hence must be equivalent to the forcing $\Add(\kappa,1)$ to add a Cohen subset to $\kappa$. Thus, the weak compactness of $\kappa$ is \emph{resurrected} in $V^{\P*\dot{\mathbb{S}}*\dot{\mathbb{C}}}$.

\begin{question}\label{question_resurrection}
Is it consistent relative to some large cardinal hypothesis on $\kappa$ that the weakly compact reflection principle $\Refl_{\wc}(\kappa)$ holds and there is a forcing $\P$ which kills $\Refl_{\wc}(\kappa)$ by adding a nonreflecting weakly compact set such that in $V^\P$ there is a forcing which resurrects $\Refl_{\wc}(\kappa)$?
\end{question}

It seems that the usual techniques for answering such questions do not work in this case. Indeed, it appears that the argument for showing that $\S*\dot{\mathbb{C}}$ contains a ${<}\kappa$-closed dense set does not generalize to the weak compactness context. Let us illustrate the difficulty.

Suppose $\kappa$ is a measurable cardinal, $W\subseteq\WC_\kappa$, $\P_{\kappa+1}\cong\P_\kappa*\dot{\Q}(\kappa,W)$ is the iteration from Theorem \ref{theorem_adding_a_nonreflectin_weakly_compact_set} for adding a nonreflecting weakly compact set and let $\dot{E}$ be a $\P_{\kappa+1}$-name for the generic nonreflecting weakly compact subset of $\kappa$. Then $\Refl_{\wc}(\kappa)$ fails in $V^{\P_{\kappa+1}}$. In order to resurrect $\Refl_{\wc}(\kappa)$, it is natural to attempt to force over $V^{\P_{\kappa+1}}$ to kill the weak compactness of $\dot{E}$. This can be done using Hellsten's \cite{MR2026390} Easton-support iteration $\mathbb{H}_{\kappa+1}$ of length $\kappa+1$, which we describe below, to shoot a $1$-club through $\kappa\setminus \dot{E}$ and preserve the weak compactness of $\kappa$. Indeed, this kills the weak compactness of $\dot{E}$ by Sun's characterization of weakly compact sets mentioned in Section \ref{sectionintroduction} which states that $W\subseteq\kappa$ is weakly compact if and only if $W\cap C\neq\emptyset$ for every $1$-club $C\subseteq\kappa$. Working in $V^{\P_{\kappa+1}}$, let $\dot{\mathbb{H}}_{\kappa+1}\cong\dot{\mathbb{H}}_{\kappa}*\dot{T}^1(\kappa\setminus\dot{E})$ be Hellsten's forcing  for shooting a $1$-club through the complement of $\dot{E}$ such that nontrivial forcing occurs at stage $\gamma$ in $\dot{\mathbb{H}}_{\kappa+1}$ if and only if nontrivial forcing occured at stage $\gamma$ in $\P_{\kappa+1}$. We would like to prove that $\P_{\kappa+1}*\dot{\mathbb{H}}_{\kappa+1}$ is forcing equivalent to the Easton-support iteration $\S_{\kappa+1}\cong\S_\kappa*\Add(\kappa,1)$ which adds a single Cohen subset to every $\gamma\leq\kappa$ at which nontrivial forcing occurs. \emph{If} this could be done then standard arguments show that since $\kappa$ is measurable in $V$, then $\kappa$ remain measurable in $V^{\P_{\kappa+1}*\dot{\H}_{\kappa+1}}=V^{\mathbb{S}_{\kappa+1}}$, and hence $\Refl_{\wc}(\kappa)$ has been resurrected.

However, when attempting to show that $\P_{\kappa+1}*\dot{\mathbb{H}}_{\kappa+1}\cong \S_{\kappa+1}$ we encountered what seems to be a serious problem with this method. To explain the issue we must describe Hellsten's forcing in more detail. The main use of Hellsten's forcing is the following.

\begin{theorem}[Hellsten, \cite{MR2026390}]\label{theorem_1_club_shooting}
Suppose $E$ is a weakly compact subset of $\kappa$. There is a forcing extension in which $E$ contains a $1$-club, all weakly compact subsets of $E$ remain weakly compact and thus $\kappa$ remains a weakly compact cardinal.
\end{theorem}

Let $X\subseteq\kappa$ be an unbounded subset of an inaccessible cardinal $\kappa$. We define a poset $T^1(X)$ as follows. Conditions in $T^1(X)$ are all $c\subseteq X$ such that $c$ is bounded and $1$-closed, meaning that for every inaccessible $\alpha<\kappa$ if $c\cap\alpha$ is stationary in $\alpha$ then $\alpha\in c$. The ordering on $T^1(X)$ is by end extension: $c_1\leq c_2$ iff $c_2=c_1\restrict\sup\{\alpha+1\st\alpha\in c_2\}$. 
Hellsten proved \cite[Lemma 3]{MR2653962} that $T^1(X)$ is $\kappa$-strategically closed. The forcing $\H_{\kappa+1}$ which Hellsten used to prove Theorem \ref{theorem_1_club_shooting} is an Easton-support iteration $\<\H_\alpha,\dot{\C}_\beta\st \alpha\leq\kappa+1,\beta\leq\kappa\>$ such that
\begin{enumerate}
\item if $\beta\leq\kappa$ is inaccessible and $E\cap\beta$ is unbounded in $\beta$ then $\dot{\C}_\beta$ is an $\H_\beta$-name for $T^1(E\cap\beta)^{V^{\H_\beta}}$ and
\item otherwise $\dot{\C}_\beta$ is an $\H_\beta$-name for trivial forcing.
\end{enumerate}

In order to show that $\P_{\kappa+1}*\dot{\H}_{\kappa+1}\cong\S_{\kappa+1}$, we would need to prove the following:

\begin{quote}
($*$) Suppose $\gamma$ is an inaccessible cardinal, $W\subseteq\WC_\gamma$ and $\gamma$ is a limit point of $\Tr_{\wc}(W)$. Let $\dot{E}$ be the canonical $\Q(\gamma,W)$-name for the subset of $W$ added by forcing with $\Q(\gamma,W)$. Then $\Q(\gamma,W)*\dot{T}^1(\gamma\setminus \dot{E})$ contains a ${<}\gamma$-closed dense set, and is hence forcing equivalent to $\Add(\gamma,1)$.
\end{quote}

Let us explain why standard methods do not seem to establish ($*$). Let $D$ be the set of conditions $(e,\dot{c})\in \Q(\gamma,W)*T^1(\gamma\setminus\dot{E})$ such that $e\forces_{\Q(\gamma,W)}$ ``$\dom(e)=\sup(\dot{c})=\alpha+1\in\Tr_{\wc}(W)$ and $\supp(e)\cap \dot{c}\neq\emptyset$.'' It is straightforward to show that $D$ is a dense subset of $\Q(\gamma,W)*\dot{T}^1(\gamma\setminus\dot{E})$. However, it is not known whether $D$ is ${<}\kappa$-closed. Suppose $\<(e_\alpha,\dot{c}_\alpha)\st\alpha<\mu\>$ is a strictly decreasing sequence of conditions in $D$ where $\mu<\gamma$. Since $\Q(\gamma,W)$ is ${<}\gamma$-distributive, we may view each $\dot{c}_\alpha$ as $c_\alpha\in V$. The only nontrivial case occurs when $\mu=\bigcup_{\alpha<\mu}\dom(e_\alpha)=\sup\bigcup_{\alpha<\mu}c_\alpha$ and $W\cap \mu$ is weakly compact. To find a lower bound of the sequence one would like to let $e=\bigcup_{\alpha<\mu}e_\alpha\cup\{(\mu,0)\}$ and then hope that $\bigcup_{\alpha<\mu}c_\alpha$ is a $1$-club subset of $\mu$ disjoint from $e$, so that $\supp(e)$ is a non--weakly compact subset of $\mu$ and hence $e$ is a condition in $\Q(\gamma,W)$. However, in order for $\bigcup_{\alpha<\mu}c_\alpha$ to be a $1$-club subset of $\mu$, it must be stationary in $\mu$, and there seems to be no reason to expect that this is the case. Thus ($*$) and Question \ref{question_resurrection} remain open.


\end{document}